\documentclass{amsart}

\usepackage{amsfonts,amsmath,amssymb} 
\usepackage[mathcal]{euscript}

\usepackage{graphicx}
\usepackage[T1]{fontenc}
\usepackage{enumerate}
\usepackage{tikz}
\usetikzlibrary{calc}
\usepackage{textcomp}  
\newtheorem{theorem}{Theorem}[section]
\newtheorem{lemma}{Lemma}[section]

\theoremstyle{definition}
\newtheorem{definition}{Definition}[section]
\newtheorem{property}{Property}
\newtheorem{example}{Example}[section]

\DeclareMathOperator{\arsinh}{arsinh}
\DeclareMathOperator{\artanh}{artanh}

\begin{document}
\title{On harmonic representation of means}

\author{Alfred Witkowski}
\address{Institute of Mathematics and Physics\\University of Technology and Life Sciences\\Al. prof. Kaliskiego 7\\85-796 Bydgoszcz, Poland}
\email{alfred.witkowski@utp.edu.pl}
\date{October 10, 2013}
\begin{abstract}
We characterize  continuous, symmetric and homogeneous  means $M$ that can be represented in the form
\begin{equation*}
\frac{1}{ M(x,y)}=\int_0^1 \frac{dt}{N\left(\tfrac{x+y}{2}-t\tfrac{x-y}{2},\tfrac{x+y}{2}+t\tfrac{x-y}{2}\right)}.
\end{equation*} 
New inequalities for means are derived from such representation.

\end{abstract}
\maketitle
\section{Introduction, Definitions and  notation}

In paper \cite{AW2} we investigated the representation of a symmetric, homogeneous mean $M:\mathbb{R}_{+}^2\to\mathbb{R}$ of the form 
\begin{equation}
M(x,y)=\frac{|x-y|}{2f\left(\frac{|x-y|}{x+y}\right)}
\label{eq:main}
\end{equation}
The main observation was that every symmetric, homogeneous mean admits such a representation. The mapping
\begin{equation}
M(x,y)\leftrightarrow f_M(z)=\frac{z}{M(1-z,1+z)}
\label{eq:mapping}
\end{equation}
establishes one-to-one correspondence between the set of symmetric homogeneous means and the set of functions $f:(0,1)\to\mathbb{R}$ satisfying 
\begin{equation}
\frac{z}{1+z}\leq f(z)\leq \frac{z}{1-z},
\label{eq:ineqmain}
\end{equation}
called \textbf{Seiffert functions}, and the identity
\begin{equation}
M(x,y)=\frac{|x-y|}{2f_M\left(\frac{|x-y|}{x+y}\right)}
\label{eq:mainM}
\end{equation}
holds. Moreover, the formula \eqref{eq:main}  transforms Seiffert function into a symmetric, homogeneous mean. \\
Note that the outermost functions in \eqref{eq:ineqmain} correspond to $\max$ and $\min$ means.

In this note we discuss the representation of means in the form
$$\frac{1}{M(x,y)}=\int_0^1\frac{dt}{N\left(\tfrac{x+y}{2}+t\tfrac{x-y}{2},\tfrac{x+y}{2}-t\tfrac{x-y}{2}\right)},$$
where $N$ is also a homogeneous, symmetric mean.

We shall be using two facts from \cite{AW2}
\begin{property}\cite[Section 7]{AW2}\label{prop:1}
	If $f$ is a Seiffert mean, then for arbitrary $0<t\leq 1$ the function $f^{\{t\}}$ given by the formula $f^{\{t\}}(z)=\frac{f(tz)}{t}$ is also a Seiffert mean.
\end{property}
\begin{lemma}
	If $f$ is a Seiffert function corresponding to the mean $M$, then $f^{\{t\}}$ is a Seiffert function for 
	$$M^{\{t\}}(x,y)=M\left(\tfrac{x+y}{2}+t\tfrac{x-y}{2},\tfrac{x+y}{2}-t\tfrac{x-y}{2}\right).$$
\end{lemma}
\begin{proof}
	Let $z=\frac{|x-y|}{x+y}$. Then by  \eqref{eq:main} and \eqref{eq:mapping} we have
	\begin{align*}\label{eq:}
		\frac{|x-y|}{2f^{\{t\}}(z)}&=\frac{t|x-y|}{2f(tz)}=	\frac{t|x-y|M(1-tz,1+tz)}{2tz}\\
		&	=\frac{x+y}{2}M\left(1-t\frac{|x-y|}{x+y},1+t\frac{|x-y|}{x+y}\right)=M^{\{t\}}(x,y).\qedhere
	\end{align*}
\end{proof}

Following \cite[Section 5]{AW2}, consider the integral operator  on the set of continuous Seiffert functions, defined as
\begin{equation}
I(f)(z)=\int_0^z\frac{f(u)}{u}du.
\label{eq:IO}
\end{equation}
\begin{property}
	The operator $I$ has the following properties:
	\begin{itemize}
		\item is monotone - if $f\leq g$, then $I(f)\leq I(g)$,
		\item preserves convexity - if $f$ is convex, then so is $I(f)$ and for all $0<z<1$ the inequalities $z\leq I(f)(z)\leq f(z)$ hold, (\cite[Theorem 5.1]{AW2}),
		\item preserves concavity - if $f$ is concave, then so is $I(f)$ and for all $0<z<1$ the inequalities $z\geq I(f)(z)\geq f(z)$ hold, (\cite[Theorem 5.1]{AW2}),
		\item $I(f)$ is a Seiffert function, (\cite[Corollary 5.1]{AW2}).
	\end{itemize}
\end{property}

The next simple theorem characterizes the functions, which are of the form $I(f)$.
\begin{theorem}\label{thm:condition to be I(f)}
	Let $g$ be a real function defined on the interval $(0,1)$. The following conditions are equivalent
	\begin{itemize}
		\item $\lim_{z\to 0} g(z)=0$, $g$ is continuously differentiable, and for all $0<z<1$
		\begin{equation}
		\frac{1}{1+z}\leq g'(z) \leq \frac{1}{1-z},
		\label{ineq:g'}
		\end{equation}
		\item there exist a continuous Seiffert function $f$ such that $g=I(f)$.
	\end{itemize}
\end{theorem}
\begin{proof}
	Multiplying \eqref{ineq:g'} by $z$ we see that $f(z)=zg'(z)$ is a continuous Seiffert function and clearly $I(f)=g$.\\
Conversely, if $f$ is continuous, then $g=I(f)$ is differentiable. Since $\lim_{z\to 0} f(z)/z=1$ we claim $\lim_{z\to 0} g(z)=0$. Differentiating $g$ we obtain $g'(z)=f(z)/z$,  which yields \eqref{ineq:g'} because $f$ fulfills \eqref{eq:ineqmain}.
\end{proof}

Now we are ready to formulate the main result of this note.
\section{Harmonic representation of means}
\begin{definition}
	We say that a continuous mean $N$ is a harmonic representation of mean $M$ if 
	$$\frac{1}{M(x,y)}=\int_0^1 \frac{dt}{N^{\{t\}}(x,y)}.$$
\end{definition}
\begin{theorem}\label{thm:condition to have harm repr}
	A continuous mean $M$ admits a harmonic representation if and only if its Seiffert function $m$ can be represented as $I(n)$, where $n$ is a continuous Seiffert function. 
\end{theorem}
\begin{proof}
	Let $N$ be the harmonic representation of $M$ and let $z=\frac{|x-y|}{x+y}$. Denote by $m$ and $n$ the Seiffert functions of $M$ and $N$ respectively. Applying \eqref{eq:main} and \eqref{eq:mapping} we have
	\begin{align*}
	\frac{2}{|x-y|}I(n)(z)	&=\frac{2}{|x-y|}\int_0^z\frac{n(u)}{u}du=\frac{2}{|x-y|}\int_0^1n^{\{t\}}(z)dt	\\
		&	=\int_0^1 \frac{dt}{N^{\{t\}}(x,y)}= \frac{1}{M(x,y)}=\frac{2}{|x-y|}m(z),
	\end{align*}
which yields $m=I(n)$.
Conversely, if $m=I(n)$ and $N$ is a mean corresponding to $n$, then 
	\begin{align*}
	\frac{1}{M(x,y)}&=\frac{2}{|x-y|}m(z)=\frac{2}{|x-y|}I(n)(z)	=\frac{2}{|x-y|}\int_0^z\frac{n(u)}{u}du\\
	&=\frac{2}{|x-y|}\int_0^1n^{\{t\}}(z)dt	
		=\int_0^1 \frac{dt}{N^{\{t\}}(x,y)}.
	\end{align*}
\end{proof}
From \eqref{eq:ineqmain} we obtain by integration the inequalities
\begin{equation}
\log (1+z)\leq I(f)(z)\leq -\log (1-z),
\label{eq:ineqI}
\end{equation}
which shows, that every mean admitting harmonic representation satisfies the inequalities
$$\frac{|x-y|}{2(\log A(x,y)-\log\min(x,y))}\leq M(x,y)\leq \frac{|x-y|}{2(\log\max(x,y)-\log A(x,y))}.$$ 
The inverse statement is not true. It is easy to construct a function satisfying \eqref{eq:ineqI} for which  \eqref{ineq:g'} fails.

\section{Examples I}
\begin{example}\label{ex:1}
	The Seiffert function of the Seiffert mean $P(x,y)=\frac{|x-y|}{2\arcsin z}$ is obviously $\arcsin$. Let $g(z)=\frac{z}{\sqrt{1-z^2}}$. Then $\arcsin=I(g)$ and $g$ is the Seiffert function of the geometric mean $G(x,y)=\sqrt{xy}$. Thus we obtain the identity
	$$P(x,y)=\left(\int_0^1\frac{dt}{G^{\{t\}}(x,y)}\right)^{-1}.$$
\end{example}
\begin{example}\label{ex:2}
	The second Seiffert mean is given by $T(x,y)=\frac{|x-y|}{2\arctan z}$. Let $C(x,y)=\frac{x^2+y^2}{x+y}$ be the contra-harmonic mean. Its Seiffert function is $c(z)=\frac{z}{1+z^2}$ and one can easily verify that $I(c)=\arctan$, so
		$$T(x,y)=\left(\int_0^1\frac{dt}{C^{\{t\}}(x,y)}\right)^{-1}.$$
\end{example}
\begin{example}\label{ex:3}
	For the logarithmic mean $L(x,y)=\frac{x-y}{\log x-\log y}=\frac{|x-y|}{2\artanh z}$ we get
	$$L(x,y)=\left(\int_0^1\frac{dt}{H^{\{t\}}(x,y)}\right)^{-1},$$
	where $H(x,y)=\frac{2xy}{x+y}$ denotes the harmonic mean.
\end{example}
\begin{example}\label{ex:4}
	The Seiffert function of the root-mean square $R=\sqrt{\frac{x^2+y^2}{2}}$ is the function $r(z)=\frac{z}{\sqrt{1+z^2}}$, thus $I(r)(z)=\arsinh z$, which in turn is the Seiffert mean of the Neuman-S\'andor mean $M(x,y)=\frac{|x-y|}{2\arsinh z}$, so
		$$M(x,y)=\left(\int_0^1\frac{dt}{R^{\{t\}}(x,y)}\right)^{-1},$$
\end{example}
In \cite{AW2} we have shown that $\sin$, $\tan$, $\sinh$ and $\tanh$ are also Seiffert function. Let us check if their corresponding means admit harmonic representations. To do it we shall use Theorems \ref{thm:condition to be I(f)} and \ref{thm:condition to have harm repr}
\begin{example}\label{ex:5}
	For $g(z)=\sin z$ we want to show that $g'$ satisfies \eqref{ineq:g'}. Obviously $\cos z<1<1/(1-z)$. To prove the other part observe that
\begin{align*}
(1+z)\cos z>&(1+z)(1-z^2/2)>1+z(1-z/2)>1,
\end{align*}	
thus \eqref{ineq:g'} holds, and one easily verifies that $z\cos z$ is the Seiffert function of the mean $M(x,y)=A(x,y)/\cos \frac{|x-y|}{x+y}$, which implies
$$\frac{x-y}{2\sin\frac{x-y}{x+y}}=\left(\int_0^1\frac{dt}{M^{\{t\}}(x,y)}\right)^{-1}.$$
\end{example}
\begin{example}\label{ex:6}
	Now let $g(z)=\tan z$. We have
	\begin{align*}
	\frac{1}{1+z}<1&<\frac{1}{\cos^2 z}=\frac{1}{(1+\sin z)(1-\sin z)}<\frac{1}{1-z},	
	\end{align*}
	so $z/\cos^2 z$ is the Seiffert function. It corresponds to the mean $M(x,y)=A(x,y)\cos^2  \frac{|x-y|}{x+y}$ and 
	$$\frac{x-y}{2\tan\frac{x-y}{x+y}}=\left(\int_0^1\frac{dt}{M^{\{t\}}(x,y)}\right)^{-1}.$$
\end{example}
\begin{example}\label{ex:7}
	With the hyperbolic sine the situation is simple. We have
	$$1<\cosh z=\sum_{m=0}^\infty\frac{z^{2m}}{(2m)!}<\sum_{m=0}^\infty z^m=\frac{1}{1-z},$$
	thus $z\cosh z$ is the Seiffert function, and its  mean $M(x,y)=A(x,y)/\cosh  \frac{|x-y|}{x+y}$ satisfies
$$\frac{x-y}{2\sinh\frac{x-y}{x+y}}=\left(\int_0^1\frac{dt}{M^{\{t\}}(x,y)}\right)^{-1}.$$
\end{example}
\begin{example}
	The last function is the hyperbolic tangent. Its derivative is $\cosh^{-2}z$ and $\cosh^{-2}(1)\approx 0.41997<\frac{1}{2}$, so the left inequality in \eqref{ineq:g'} does not hold, and this yields the mean $\frac{x-y}{2\sinh\frac{x-y}{x+y}}$ does not have a harmonic representation.
\end{example}
We leave as a simple exercise the fact that there is no harmonic representation of the geometric mean.
\section{The arithmetic-geometric mean \label{sec:AGM}}
 This section is devoted to the arithmetic-geometric mean given by the formula
$$AGM(x,y)=\left(\frac{2}{\pi}\int_0^{\pi/2}\frac{d\varphi}{\sqrt{x^2\cos^2\varphi+y^2\sin^2\varphi}}\right)^{-1}.$$
To find its Seiffert mean let us recall the famous result of Gauss \cite{Ga}
\begin{equation}
AGM(1-z,1+z)=\frac{\pi}{2K(z)},
\label{eq:Gaussidentity}
\end{equation}
where $K$ is the complete elliptic integral of the first kind
\begin{equation}
K(z)=\int_0^{\pi/2}\frac{d\varphi}{\sqrt{1-z^2\sin^2\varphi}}=\int_0^1\frac{dt}{\sqrt{1-t^2}\sqrt{1-z^2t^2}}.
\label{eq:K}
\end{equation}
Comparing \eqref{eq:Gaussidentity} and \eqref{eq:mapping} we see that $f_{AGM}(z)=\frac{2}{\pi}zK(z)$. We shall show that $AGM$ admits the harmonic representation. By Theorem \ref{thm:condition to be I(f)} it is enough to show that $f_{AGM}'$satisfies \eqref{ineq:g'}. To this end let us recall the power series expansion of $K$ (\cite[900.00]{BF})
\begin{equation}
K(z)=\frac{\pi}{2}\left(1+\sum_{m=1}^\infty \left[\frac{(2m-1)!!}{(2m)!!}\right]^2 z^{2m}\right).
\label{eq:powerseries for K}
\end{equation}
We have
\begin{align}\label{eq:f_AGM"' series}
	f_{AGM}'(z)&=\frac{2}{\pi}\left(K(z)+z\frac{dK}{dz}\right)=1+\sum_{m=1}^\infty (2m+1)\left[\frac{(2m-1)!!}{(2m)!!}\right]^2 z^{2m}	
\end{align}
Denoting the $m^{\rm{th}}$ coefficient in \eqref{eq:f_AGM"' series} by $c_m$ we see that
$$\frac{c_{m+1}}{c_m}=\frac{2m+3}{2m+1}\left[\frac{(2m+1)!!((2m)!!}{(2m+2)!!(2m-1)!!}\right]^2=\frac{(2m+1)(2m+3)}{(2m+2)^2}<1,$$
and since $c_1=3/4$ we conclude that $c_m<1$ for all $\geq 1$. Thus $1<f_{AGM}'(z)<1+z+z^2+\dots=1/(1-z)$.\\
Theorem \ref{thm:condition to be I(f)} implies that the arithmetic-geometric mean admits the harmonic representation. To derive its explicit form, recall  that the derivative of $K$ is given by  $K'(z)=\frac{E(z)}{z(1-z^2)}-\frac{K(z)}{z} $(see. e.g. \cite[710.00]{BF}), thus
\begin{align*}\label{eq:}
	zf_{AGM}'(z)&=\frac{2}{\pi}\left(zK(z)+z^2K'(z)\right)=\frac{2}{\pi}\frac{z}{1-z^2}E(z),
\end{align*}
($E(z)=\int_0^{\pi/2}\sqrt{1-z^2\sin^2\varphi}d\varphi$ is the complete elliptic integral of the second kind). 
As $\frac{z}{1-z^2}$ is the Seiffert function of the harmonic mean we obtain the formula
\begin{align*}
	V(x,y)=&\frac{\pi H(x,y)}{2E\left(\frac{|x-y|}{x+y}\right)}=\frac{\pi H(x,y)}{2E\left(\sqrt{1-\frac{G^2(x,y)}{A^2(x,y)}}\right)}\\=&	\frac{\pi G^2(x,y)}{2\int_0^{\pi/2}\sqrt{A^2(x,y)\cos^2\varphi + G^2(x,y)\sin^2\varphi}d\varphi}.
\end{align*} 
This mean has a nice geometric interpretation: in the ellipsis with semi-axes $G(x,y)$ and $A(x,y)$ it represents the ratio of the area of inscribed disc to its semi-perimeter.

\section{Hermite-Hadamard inequality for means}
The Hermite-Hadamard inequality in its classic form says that if $f$ is a convex function in an interval $I$, then for all $a,b\in I$
$$f\left(\frac{a+b}{2}\right)\leq \frac{1}{b-a}\int_a^b f(t)dt\leq \frac{f(a)+f(b)}{2}.$$
A stronger inequality also holds
$$ \frac{1}{b-a}\int_a^b f(t)dt\leq \frac{1}{2}\left[f\left(\frac{a+b}{2}\right) +\frac{f(a)+f(b)}{2}\right].$$

Suppose now that the mean $N$ is the harmonic representation of $M$ and its Seiffert function $n$ is such that the function $n(u)/u$ is convex. Then, applying the Hermite-Hadamard inequality to \eqref{eq:IO} and taking into account that $\lim_{u\to 0} n(u)/u=1$ we obtain
\begin{equation}
2n(z/2)\leq I(n)(z)\leq \frac{z+n(z)}{2}.
\label{eq:HHI1}
\end{equation}
This yields (with help of  \eqref{eq:mapping})  the inequalities for means
\begin{equation}
H(A(x,y),N(x,y))\leq M(x,y) \leq N\left(\frac{3x+y}{4},\frac{x+3y}{4}\right).
\label{eq:HH for means}
\end{equation}

The stronger version of the Hermite-Hadamard reads in this case:
\begin{equation}
 I(n)(z)\leq \frac{1}{2}\left[2n(z/2)+\frac{z+n(z)}{2}\right],
\label{eq:HHI2}
\end{equation}
which yields
\begin{equation}
 H(A(x,y),N^{\{1/2\}}(x,y),N^{\{1/2\}}(x,y),N(x,y))\leq M(x,y) \leq N\left(\frac{3x+y}{4},\frac{x+3y}{4}\right).
\label{eq:HH2 for means}
\end{equation}
Obviously, if $n(u)/u$ is concave, the inequalities in \eqref{eq:HHI1}--\eqref{eq:HH2 for means} are reversed.

In the above we use the Hermite-Hadamard inequality with the left end fixed, so it may happen that \eqref{eq:HHI1} holds even if $n(u)/u$ is not convex. Of course, in such case an individual treatment would be required.

\section{Examples II}
\begin{example}
Let $N=G$. By Example \ref{ex:1} we know that $M=P$ is the first Seiffert mean. Since $n(u)/u=(1-u^2)^{-1/2}$ is convex and $G^{\{1/2\}}=\sqrt{3A^2+G^2}/2$, \eqref{eq:HH for means} and \eqref{eq:HHI2} yield
$$\frac{2AG}{A+G}\leq 2\left(\frac{2}{\sqrt{3A^2+G^2}}+\frac{A+G}{2}\right)^{-1}\leq P \leq \frac{\sqrt{3A^2+G^2}}{2}.$$
\end{example}
\begin{example}The Seiffert function $c$  from Example \ref{ex:2} does not satisfy the convexity condition, but the reversed inequalities in \eqref{eq:HHI1} hold anyway, by the following lemma.
\begin{lemma}\label{l:ex2}
	The inequalities 
	$$\frac{4u}{4+u^2}>\arctan u > u\frac{2+u^2}{2+2u^2}$$
	hold for $0<u<1$ 
\end{lemma}
\begin{proof}
	Let $h(u)=\frac{4u}{4+u^2}-\arctan u$. As $h(0)=0$ and $h'(u)=\frac{u^2(4-5u^2)}{(u^2+1)(u^2+4)^2} $ we see that $h$ has local maximum at $u=2/\sqrt{5}$ and since $h(1)>0$ we conclude that $h(u)>0$.
	
	Let now $h(u)=\arctan u-u\frac{2+u^2}{2+2u^2}$. Then $h(0)=0$ and $h'(u)=\frac{u^2(1-u^2)}{2(x^2+1)^2}>0$, and the proof is complete.
\end{proof}
Thus for the  contraharmonic mean and the second Seiffert mean we have
$$ C^{\{1/2\}}=\frac{5A^2-G^2}{4A}\leq T\leq H(A,C)$$
	
\end{example}
\begin{example}
	The pair $(M,N)=(L,H)$ (see Example \ref{ex:3}) gives the inequalities
	$$\frac{2G^2A}{A^2+G^2}\leq \frac{4AG^2(3A^2+G^2)}{3A^4+12A^2G^2+G^4}\leq L\leq \frac{3A^2+G^2}{4A}$$
\end{example}
 
\begin{example}
For the root-mean square and Neuman-S\'andor means (Example \ref{ex:4}) the convexity condition is not satisfied, but the following lemma shows that the reversed inequalities \eqref{eq:HHI1} are valid.
\begin{lemma}
	For $0<u<1$ the inequalities
	$$\frac{2u}{\sqrt{u^2+4}}\geq \arsinh u\geq \frac{u}{2}+\frac{u}{2\sqrt{u^2+1}} $$ 
	hold.
\end{lemma}
\begin{proof}
	To prove the left inequality it suffices to show that the function $h(u)=\arsinh u-\frac{2u}{\sqrt{u^2+4}}$ decreases, because $h(0)=0$. Differentiating we obtain
	\begin{equation}
	h'(u)=\frac{(u^2+4)^{3/2}-8(u^2+1)^{1/2}}{(u^2+4)^{3/2}(u^2+1)^{1/2}}.
	\label{eq:h_1}
	\end{equation}

Let $p$ denote the numerator in \eqref{eq:h_1}. Then $p'(u)=u\left(3\sqrt{u^2+4)}-\frac{8}{\sqrt{u^2+1}}\right):=uq(u)$. The function $q$ is a difference of an increasing and decreasing function, thus increases from $q(0)=-2$ to $q(1)=3\sqrt{5}-4\sqrt{2}>0$, so we conclude that $p$ has one local minimum in the interval $(0,1)$. Since $p(0)=0$ and $p(1)=\sqrt{125}-\sqrt{128}<0$ we see that $p(u)<0$ for all $u$, thus $h'(u)<0$ and we are done.

For the right inequality the method is similar:
\begin{align*}
h(u)=\frac{u}{2}+\frac{u}{2\sqrt{u^2+1}}-\arsinh u,	&\quad h'(u)=\frac{(u^2+1)^{3/2}-(2u^2+1)}{2(u^2+1)^{3/2}}	\\
p(u)=(u^2+1)^{3/2}-(2u^2+1),	&\quad p'(u)=u(3\sqrt{u^2+1}-4):=uq(u).	
\end{align*}
As above, $q$ increases from $-1$ to $3\sqrt{2}-4$, so $p$ has one local minimum, and since $p(0)=0$ and $p(1)=\sqrt{8}-3<0$ we conclude $h'<0$.
\end{proof}
Thus for the Neuman-S\'andor mean $M(x,y)=\frac{|x-y|}{2\arsinh \frac{|x-y|}{x+y}}$ the inequality \eqref{eq:HH for means} in this case reads
$$R^{\{1/2\}}=\frac{\sqrt{5A^2-G^2}}{2}\leq M \leq H(A,R).$$
\end{example}
\begin{example}
	In Example \ref{ex:5} we consider the Seiffert functions $m(z)=\sin z$ and $n(z)=z\cos z$. Clearly $n(z)/z$ is concave and thus
	$$\frac{x+y}{2\cos \frac{1}{2}\frac{|x-y|}{x+y}}\leq \frac{|x-y|}{2\sin \frac{|x-y|}{x+y} }\leq \frac{x+y}{1+\cos \frac{|x-y|}{x+y}}$$
\end{example}
\begin{example}
	The function $\frac{1}{\cos^2 z}$ is convex, thus we can apply \eqref{eq:HHI1} to the functions from Example \ref{ex:6} to obtain  
	$$\frac{(x+y)\cos^2 \frac{|x-y|}{x+y}}{1+\cos^2 \frac{|x-y|}{x+y}}\leq \frac{|x-y|}{2\tan \frac{|x-y|}{x+y} }\leq {A(x,y)}{\cos^2 \frac{1}{2}\frac{|x-y|}{x+y}}.$$
\end{example}
\begin{example}
	In Example \ref{ex:7} the function $\cosh$ is convex, so we get
	$$\frac{x+y}{1+\cosh \frac{|x-y|}{x+y}}\leq \frac{|x-y|}{2\sinh \frac{|x-y|}{x+y} }\leq \frac{x+y}{2\cosh \frac{1}{2}\frac{|x-y|}{x+y}}.$$
\end{example}

\begin{example}
	In this example we deal with the $AGM$ mean and its harmonic representation $V$ described in Section \ref{sec:AGM}. The Seiffert mean of $V$ is $v(z)=\frac{2}{\pi}\frac{z}{1-z^2}E(z)$, so
	\begin{equation}
	\frac{v(z)}{z}=\frac{2}{\pi}\int_0^{\pi/2} \frac{\sqrt{1-z^2\sin^2\varphi}}{1-z^2}d\varphi.
	\label{eq:v}
	\end{equation}
We shall show that this function is convex. For $0<a<1$ let $h_a(u)=\frac{\sqrt{1-au^2}}{\sqrt{1-u^2}}$. Then
$$h_a'(u)=\frac{(1-a)u}{(1-au^2)^{1/2}(1-u^2)^{3/2}}.$$
Note the $h_a'$ is nonnegative and increasing, since its numerator increases while denominator decreases. Thus $h_a$ is positive, increasing and convex. The function $g(u)=1/\sqrt{(1-u^2)}$ shares the same properties, so their product is convex \cite[Theorem I.13C]{VR}. Since the integrands in \eqref{eq:v} are convex, so is the left-hand side.
Therefore by \eqref{eq:HH for means}
$$\frac{2AV}{A+V}\leq AGM\leq V^{\{1/2\}}.$$
\end{example}

==================================================================================================================
==================================================================================================================

\end{document}